\documentclass[12pt]{article}

\usepackage{amsmath,amsthm,amsfonts,amssymb,amscd}

\allowdisplaybreaks[4]

\newtheorem {Lemma}{Lemma}[section]
\newtheorem {Theorem} {Theorem}[section]

\newtheorem {Corollary}{Corollary}[section]

\allowdisplaybreaks [4]
\headsep =
0.0in \headheight = 0.0in \topmargin = 0.3in

\begin{document}

\title{On the distance $\alpha$-spectral radius of a connected graph}

\author{Haiyan Guo\footnote{ghaiyan0705@163.com}, Bo Zhou\footnote{Corresponding author. E-mail: zhoubo@scnu.edu.cn}\\
School of  Mathematical Sciences, South China Normal University,\\
Guangzhou 510631, P.R. China}

\date{}
\maketitle

\begin{abstract}
For a connected graph $G$ and $\alpha\in [0,1)$, the distance $\alpha$-spectral radius of $G$ is the spectral radius of the matrix $D_{\alpha}(G)$ defined as $D_{\alpha}(G)=\alpha T(G)+(1-\alpha)D(G)$, where $T(G)$ is a diagonal matrix of vertex transmissions of $G$ and $D(G)$ is the distance matrix of $G$.  We give bounds for the distance $\alpha$-spectral radius, especially for graphs that are not transmission regular,   propose some  graft transformations that decrease or increase the distance $\alpha$-spectral radius, and determine the unique graphs with minimum and maximum distance $\alpha$-spectral radius among some classes of graphs. \\ \\
{\bf AdMS classifications:} 05C50, 15A18
\end{abstract}

\section{Introduction}

We consider simple and undirected graphs. Let $G$ be a connected graph of order $n$ with vertex set $V(G)$ and edge set $E(G)$. For $u,v\in V(G)$, the distance between $u$ and $v$ in $G$, denoted by $d_G(u,v)$ or simply $d_{uv}$, is the length of a shortest path from $u$ to $v$ in $G$. The distance matrix of $G$ is the $n\times n$ matrix
$D(G)=(d_G(u,v))_{u,v\in V(G)}$. For $u\in V(G)$, the transmission of $u$ in $G$, denoted by $T_G(u)$, is defined as the sum of distances from $u$ to all other vertices of $G$, i.e., $T_G(u)=\sum_{v\in V(G)}d_G(u,v)$. The transmission matrix $T(G)$ of $G$ is the diagonal matrix of transmissions of $G$. Then $Q(G)=T(G)+D(G)$ is the distance signless Laplacian matrix of $G$.

Throughout this paper we assume that $\alpha\in [0,1)$. We consider the convex combinations $D_{\alpha}(G)$ of $T(G)$ and $D(G)$, defined as
\[
D_{\alpha}(G)=\alpha T(G)+(1-\alpha)D(G).\]
Obviously, $D_0(G)=D(G)$ and $2D_{1/2}(G)=Q(G)$. We call the eigenvalues of $D_{\alpha}(G)$ the distance $\alpha$-eigenvalues of $G$.  As $D_{\alpha}(G)$ is a symmetric matrix, the distance $\alpha$-eigenvalues of $G$ are all real, which are denoted by $\mu^{(1)}_{\alpha} (G), \dots, \mu^{(n)}_{\alpha}(G)$, arranged in nonincreasing order, where $n=|V(G)|$. The largest distance $\alpha$-eigenvalue  $\mu^{(1)}_{\alpha} (G)$ of $G$ is called the distance $\alpha$-spectral radius of $G$, written as $\mu_{\alpha} (G)$.
Obviously, $\mu^{(1)}_{0} (G), \dots, \mu^{(n)}_{0}(G)$ are the distance eigenvalues of $G$, and $2\mu^{(1)}_{1/2} (G), \dots, 2\mu^{(n)}_{1/2}(G)$ are the distance signless Laplacian eigenvalues of $G$. Particularly,  $\mu_0(G)$ is just the distance spectral radius and  $2\mu_{1/2}(G)$ is just the distance signless Laplacian spectral radius of $G$.
The distance eigenvalues and especially the distance spectral radius have been extensively studied, see the recent survey~\cite{AH} and references therein. The distance signless Laplacian eigenvalues and especially the distance signless Laplacian spectral radius have also received much attention, see, e.g., \cite{AH2, AP, BKNS, DAS, LL, LZ, XZL}.

In this paper,  we give sharp bounds for the distance $\alpha$-spectral radius, and particularly an upper bound  for the distance $\alpha$-spectral radius of connected graphs that are not transmission regular,  and propose some types of graft transformations that decrease or increase the distance $\alpha$-spectral radius. We also   determine the unique graphs with minimum  distance $\alpha$-spectral radius among trees and unicyclic graphs, respectively, as well as the unique graphs (trees) with maximum and second maximum distance $\alpha$-spectral radii, and the unique graph with maximum  distance $\alpha$-spectral radius among connected graphs with given clique number, and among odd-cycle  unicyclic graphs, respectively.

\section{Preliminaries}

Let $G$ be a connected graph with $V(G)=\{v_1,\dots,v_n\}$. A column vector $x=(x_{v_1},\dots, x_{v_n})^\top\in \mathbb{R}^n$ can be considered as a function defined on $V(G)$ which maps vertex $v_i$ to $x_{v_i}$, i.e., $x(v_i)=x_{v_i}$ for $i=1,\dots,n$. Then
\[
x^\top D_{\alpha}(G)x=\alpha\sum_{u\in V(G)}T_G(u)x_u^2+2\sum_{\{u,v\}\subseteq V(G)}(1-\alpha)d_G(u,v)x_{u}x_{v},
\]
or equivalently,
\[
x^\top D_{\alpha}(G)x=\sum_{\{u,v\}\subseteq V(G)}d_G(u,v)\left(\alpha(x_u^2+x_v^2)+2(1-\alpha)x_{u}x_{v}\right).
\]
Since $D_{\alpha}(G) $ is a nonnegative irreducible matrix, by the Perron-Frobenius theorem, $\mu_{\alpha} (G)$ is simple and there is a unique positive unit eigenvector corresponding  to $\mu_{\alpha} (G)$, which is called the distance $\alpha$-Perron vector  of $G$.
If $x$ is the distance $\alpha$-Perron vector of $G$, then  for each $u\in V(G)$,
\[
\mu_{\alpha}(G)x_u=\alpha T_G(u)x_u+(1-\alpha)\sum_{v\in V(G)}d_G(u,v)x_v,
\]
or equivalently,
\[
\mu_{\alpha}(G)x_u=\sum_{v\in V(G)}d_G(u,v)(\alpha x_u+(1-\alpha)x_v),
\]
which  is called the $\alpha$-eigenequation of $G$ at $u$.
For a unit column vector $x\in\mathbb{R}^n$ with at least one nonnegative entry, by Rayleigh's principle, we have $\mu_{\alpha} (G)\ge x^{\top}D_{\alpha}(G)x$
with equality if and only if $x$ is the distance $\alpha$-Perron vector  of $G$.


\begin{Lemma}\label{auto}
Let $G$ be a connected graph with $\eta$ being an automorphism of $G$, and $x$ a distance $\alpha$-Perron vector of $G$. Then for $u,v\in V(G)$,  $\eta(u)=v$ implies that $x_u=x_v$.
\end{Lemma}

\begin{proof}  Let $P=(p_{uv})_{u,v\in V(G)}$ be  the permutation matrix such that  $p_{vu}=1$ if and only if  $\eta(u)=v$ for  $u,v\in V(G)$.  We have $D_{\alpha}(G) = P^{\top}D_{\alpha}(G) P$ and $Px$ is a positive unit vector. Thus $\mu_{\alpha}(G)=x^\top D_{\alpha}(G)x=(Px)^\top D_{\alpha}(G)(Px)$, implying that $Px$ is also a distance $\alpha$-Perron vector of $G$. Thus $P x = x$, and  the result follows.
\end{proof}

Let $G$ be a graph. For $v\in V(G)$, let $N_G(v)$ be the set of neighbors of $v$ in $G$, and $d_G(v)$ be the  degree of  $v$ in $G$. Let $G-v$ be the subgraph of $G$ obtained by deleting $v$ and all edges containing $v$. For $S\subseteq V(G)$, let $G[S]$ be the subgraph induced by $S$. For a subset $E'$ of $E(G)$, $G-E'$ denotes the graph obtained from $G$ by deleting all the edges in $E'$, and in particular, we write $G-xy$ instead of $G-\{xy\}$ if $E_1=\{xy\}$.
Let $\overline{G}$ be the complement of $G$. For a subset $E'$ of $E(\overline{G})$, denote $G+E'$ the graph obtained from $G$ by adding all edges in $E'$, and in particular, we write $G+xy$ instead of $G+\{xy\}$ if $E'=\{xy\}$.

For a nonnegative square  matrix $A$, the Perron-Frobenius theorem implies that $A$ has an eigenvalue that  is equal  the maximum modulus of all its  eigenvalues; this eigenvalue is called the spectral radius of $A$, denoted by $\rho(A)$. Obviously, $\mu_{\alpha}(G)=\rho(D_{\alpha}(G))$ for a connected graph $G$. 

Restating Corollary 2.1 in \cite[p.~38]{HM}, we have

\begin{Lemma}\label{non}\cite{HM}  Let $A$ and $B$ be square nonnegative matrices. If $A$ is irreducible, $A\ge B$, and $A\neq B$, then $\rho(A)> \rho (B)$.
\end{Lemma}

By Lemma~\ref{non}, we have

\begin{Lemma}\label{ad} Let $G$ be a connected graph with $u,v\in V(G)$. If $u$ and $v$ are not adjacent, then
$\mu_{\alpha}(G+uv)< \mu_{\alpha}(G)$.
\end{Lemma}

The transmission of a connected graph $G$, denoted by $\sigma (G)$, is the sum of distance between all unordered pairs of vertices in $G$. Clearly, $\sigma (G)=\frac{1}{2}\sum_{v\in V(G)} T_G(v)$. A graph is said to be transmission regular if $T_G(v)$ is a constant for each $v\in V(G)$.

\begin{Lemma}\label{xz} Let $G$ be a connected graph of order $n$. Then
\[\mu_{\alpha}(G)\ge \frac{2\sigma (G)}{n}\] with equality if and only if $G$ is transmission regular.
\end{Lemma}
\begin{proof}  Let $x=\frac{1}{\sqrt{n}}(1,1,\dots,1)^\top$. Obviously, $xx^\top=1$. Then
\[
\mu_{\alpha}(G)\ge x^\top D_{\alpha}(G)x
=\sum_{\{u,v\}\subseteq V(G)}d_G(u,v)\left(\alpha\left(x_u^2+x_v^2\right)+2(1-\alpha)x_u x_v\right)
=\frac{2\sigma (G)}{n}.
\]
Equality holds if and only if $x$ is the distance $\alpha$-Perron vector of $G$, equivalently,
\[
\mu_{\alpha}(G)x_u=\alpha T_G(u)x_u+(1-\alpha)\sum_{v\in V(G)}d_G(u,v)x_v=T_G(u)x_u \mbox{ for } u\in V(G),
\]
i.e., $T_G(u)=\mu_{\alpha}(G)$ for $u\in V(G)$. 
\end{proof}

Let $J_{s\times t}$ be the $s\times t$ matrix of all $1$'s,
 $0_{s\times t}$ the $s\times t$ matrix of all $0$'s, and $I_s$  the identity matrix of order $s$.

Let $K_n$, $P_n$, and $S_n$  be the complete graph, the path, and the star on $n$ vertices, respectively.

\section{Bounds for the distance $\alpha$-spectral radius}

In this section, we give some sharp  bounds for the distance $\alpha$-spectral radius, some of which may serve as a gentle warm-up exercise.

Note that  $D_{\alpha}(K_n)=\alpha(n-1) I_n+(1-\alpha)(J_{n\times n}-I_n)$, and thus $\mu_{\alpha}(K_n)=n-1$. By Lemma~\ref{ad}, we have

\begin{Theorem}
Let $G$ be a connected graph of order $n$. Then
\[
\mu_{\alpha}(G)\ge n-1
\]
with equality if and only if $G\cong K_n$.
\end{Theorem}

If  $(d_1,\dots, d_n)$ is the nonincreasing degree sequence of a graph $G$ of order at least $2$, then $d_1$ (resp. $d_2$) is the maximum (resp. second maximum) degree,  $d_n$ (resp. $d_{n-1}$) is the minimum (resp. second minimum) degree of $G$. The diameter of $G$ is the maximum distance between all vertex pairs of $G$.

We use the techniques from \cite{ZI}.

\begin{Theorem} \label{low}
Let $G$ be a connected graph of order $n\ge 2$ with maximum degree $\Delta$ and second maximum degree $\Delta'$. Then
\begin{eqnarray*}
\mu_{\alpha}(G) &\ge & \frac{1}{2}\left(
\vphantom{+\sqrt{\alpha^2(4n-4-\Delta-\Delta')^2-4(2\alpha-1)(2n-2-\Delta)(2n-2-\Delta')}}
\alpha(4n-4-\Delta-\Delta') \right.\\
&& \left.+\sqrt{\alpha^2(4n-4-\Delta-\Delta')^2-4(2\alpha-1)(2n-2-\Delta)(2n-2-\Delta')}\right)
\end{eqnarray*}
with equality if and only if $G$ is regular with diameter at most $2$.
\end{Theorem}

\begin{proof}
Let $x$ be the distance $\alpha$-Perron vector of $G$.  Let
\[
x_u=\min\{x_w: w\in V(G)\}  \mbox{ and } x_v=\min\{x_w: w\in V(G)\setminus\{u\}\}.
\]
From the $\alpha$-eigenequations of $G$ at $u$ and $v$, we have
\begin{eqnarray*}
(\mu_{\alpha}(G)-\alpha T_G(u))x_u &=& (1-\alpha)\sum_{w\in V(G)}d_G(u,w)x_w\\
& \ge & (1-\alpha)\sum_{w\in V(G)}d_G(u,w)x_v\\\\
&=& (1-\alpha) T_G(u)x_v
\end{eqnarray*}
and
\begin{eqnarray*}
(\mu_{\alpha}(G)-\alpha T_G(v))x_v &=& (1-\alpha)\sum_{w\in V(G)}d_G(v,w)x_w\\
& \ge & (1-\alpha)\sum_{w\in V(G)}d_G(v,w)x_u\\\\
&=& (1-\alpha) T_G(v)x_u.
\end{eqnarray*}
Thus
\[
(\mu_{\alpha}(G)-\alpha T_G(u))(\mu_{\alpha}(G)-\alpha T_G(v))\ge (1-\alpha)^2 T_G(u)  T_G(v),
\]
i.e.,
\[
\mu_{\alpha}^2(G)-\alpha(T_G(u)+T_G(v))\mu_{\alpha}(G)+(2\alpha-1)T_G(u)T_G(v)\ge 0.
\]
Note that $\mu_{\alpha}(G)>\alpha T_G(u)$, $\mu_{\alpha}(G)>\alpha T_G(v)$, and then
$\mu_{\alpha}(G)> \frac{\alpha(T_G(u)+T_G(v))}{2}$.
Thus
\[
\mu_{\alpha}(G)\ge f(T_G(u),T_G(v))
\]
with
\[
f(s,t)=\frac{\alpha(s+t)+\sqrt{\alpha^2(s+t)^2-4(2\alpha-1)st}}{2}.
\]

It is easily seen that
$T_G(u)\ge d_G(u)+ 2\cdot (n-1-d_G(u))=2n-2-d_G(u)$. Similarly,  $T_G(v)\ge 2n-2-d_G(v)$.
Assume that $d_G(u)\ge d_G(v)$. Then
\[
T_G(u)\ge 2n-2-\Delta \mbox{ and }  T_G(v)\ge 2n-2-\Delta'.
\]
Obviously, $f(s,t)$ is strictly increasing for $s, t\ge 1$. Thus
\[
\mu_{\alpha}(G)\ge f(2n-2-\Delta,2n-2-\Delta')£¬
\]
as desired.

Suppose that the lower bound for  $\mu_{\alpha}(G)$ is attained. Then all entries of $x$ are equal to $x_u$ or $x_v$, and hence are the same. Therefore all transmissions are equal, and the diameter $d$ is at most $2$. If $d=1$, then $G$ is complete. If $d=2$, then $\mu_{\alpha}(G)=T_G(w)=2n-2-d_w$ for $w\in V(G)$, and thus $G$ is regular.

Conversely, if $G$ is regular with diameter at most $2$, then $T_G(w)=2n-2-d_G(w)$ for $w\in V(G)$, and thus the lower bound for $\mu_{\alpha}(G)$ is attained.
\end{proof}

\begin{Theorem} \label{up}
Let $G$ be a connected graph of order  $n\ge 2$ with minimum degree $\delta$ and second minimum degree $\delta'$. Let $d$ be the diameter of $G$. Let
$S=dn-\frac{d(d-1)}{2}-1-\delta (d-1)$ and $S'=dn-\frac{d(d-1)}{2}-1-\delta' (d-1)$.
Then
\begin{eqnarray*}
\mu_{\alpha}(G) &\le & \frac{1}{2}\large\left(\vphantom{\sqrt{\alpha^2}} \alpha(2dn-2-(d-1)(d+\delta+\delta')) \right.\\
&& \left.+\sqrt{\alpha^2(2dn-2-(d-1)(d+\delta+\delta'))^2-4(2\alpha-1) S S'}\right)
\end{eqnarray*}
with equality if and only if $G$ is regular with $d\le 2$.
\end{Theorem}

\begin{proof}
Let $x$ be the distance $\alpha$-Perron vector of $G$.  Let
\[
x_u=\max\{x_w: w\in V(G)\}  \mbox{ and } x_v=\max\{x_w: w\in V(G)\setminus\{u\}\}.
\]
From the $\alpha$-eigenequations of $G$ at $u$ and $v$, we have
\begin{eqnarray*}
(\mu_{\alpha}(G)-\alpha T_G(u))x_u &=& (1-\alpha)\sum_{w\in V(G)}d_G(u,w)x_w\\
& \le & (1-\alpha)\sum_{w\in V(G)}d_G(u,w)x_v\\
&=& (1-\alpha) T_G(u)x_v
\end{eqnarray*}
and
\begin{eqnarray*}
(\mu_{\alpha}(G)-\alpha T_G(v))x_v &=& (1-\alpha)\sum_{w\in V(G)}d_G(v,w)x_w\\
& \le & (1-\alpha)\sum_{w\in V(G)}d_G(v,w)x_u\\\\
&=& (1-\alpha) T_G(v)x_u.
\end{eqnarray*}
Thus we have 
\[
\mu_{\alpha}^2(G)-\alpha(T_G(u)+T_G(v))\mu_{\alpha}(G)+(2\alpha-1)T_G(u)T_G(v)\le 0.
\]
Thus
\[
\mu_{\alpha}(G)\le f(T_G(u),T_G(v))
\]
with
\[
f(s,t)=\frac{\alpha(s+t)+\sqrt{\alpha^2(s+t)^2-4(2\alpha-1)st}}{2}.
\]
Assume that $d_G(u)\le d_G(v)$. Note that
\begin{eqnarray*}
T_G(u) &\le & d_G(u)+\sum_{i=2}^{d-1}i+d\left(n-1-d_G(u)-\sum_{i=2}^{d-1}1\right)\\
&=&dn-\frac{d(d-1)}{2}-1-d_G(u)(d-1)\\
&\le &  dn-\frac{d(d-1)}{2}-1-\delta (d-1)
\end{eqnarray*}
and similarly,
\[
T_G(v)\le dn-\frac{d(d-1)}{2}-1-\delta'(d-1).
\]
Since $f(s,t)$ is strictly increasing for $s, t\ge 1$, we have
\[
\mu_{\alpha}(G)\le f\left(dn-\frac{d(d-1)}{2}-1-\delta (d-1),dn-\frac{d(d-1)}{2}-1-\delta' (d-1)\right),
\]
as desired.

Suppose the upper bound for $\mu_{\alpha}(G)$ is attained. Then
 all entries of $x$ are equal, and thus all transmissions are equal. If $d\ge 3$, then from the the above argument, for every vertex $w$, there is exactly one vertex $w'$ with $d_G(w,w')=2$, and thus $d=3$, and for a vertex $z$ of eccentricity $2$,
 \[
d_G(z)+(n-1-d_G(z))\cdot 2=T_G(z)=3n-\frac{3\times(3-1)}{2}-1-(3-1)\delta,
 \]
 implying that $\delta\ge n-2$. Obviously, $G\not\cong P_4$. For a diametrical path $P=v_0v_1v_2v_3$, $v_0$ and $v_3$ should be adjacent to all vertices outside $P$, implying that $d=2$, a contradiction. Therefore $G$ is regular with $d\le 2$.

 Conversely, if $G$ is regular with $d\le 2$, then $T_G(w)=2n-2-d_G(w)$ for $w\in V(G)$, and thus the upper bound for $\mu_{\alpha}(G)$ is attained.
\end{proof}

 For an $n\times n$ nonnegative matrix $A=(a_{ij})$, let $r_i$ be the $i$-th row sum of $A$, i.e.,  $r_i=\sum_{j=1}^n a_{ij}$ for $i=1, \dots, n$, and let $r_{\min}$ and $r_{\max}$ be the minimum and maximum row sums of $A$, respectively.

\begin{Lemma}\label{ela} \cite{AP} Let $A=(a_{ij})$ be an $n\times n$ nonnegative matrix with row sums $r_1,\dots,r_n$. Let $S=\{1,\dots,n\}$, $r_{\min}=r_p$, $r_{\max}=r_q$ for some $p$ and $q$ with $1 \le p, q\le n $, $\ell=\max\{r_i-a_{ip}: i\in S\setminus\{p\}\}$, $m=\min\{r_i-a_{iq}: i\in S\setminus\{q\}\}$, $s=\max\{a_{ip}: i\in S\setminus\{p\}\}$ and $t=\min\{a_{iq}: i\in S\setminus\{q\}\}$. Then
\begin{eqnarray*}
&&\frac{a_{qq}+m+\sqrt{(m-a_{qq})^2+4t(r_{\max}-a_{qq})}}{2}\\
&\le& \rho(A)\\
&\le&\frac{a_{pp}+\ell+\sqrt{(\ell-a_{pp})^2+4s(r_{\min}-a_{pp})}}{2}.
\end{eqnarray*}
Moreover, the first equality holds if $r_i-a_{iq}=m$ and $a_{iq}=t$ for all $i\in S\setminus\{q\}$, and the second equality holds if $r_i-a_{ip}=\ell$ and $a_{ip}=s$ for all $i\in S\setminus\{p\}$.
\end{Lemma}

 A connected graph $G$ on $n$ vertices is  distinguished vertex deleted regular (DVDR)  if there is a vertex $v$  of degree $n-1$  such that  $G-v$ is regular.

\begin{Lemma} \label{DVDR} \cite{AP} Let $G$ be a non-complete connected graph of order $n$. Then $G$ is a DVDR graph if and only if
each vertex of $G$ except one vertex $v$ of degree $n-1$ has the same transmission.
\end{Lemma}

For a connected graph $G$, let $T_{\min}(G)$ and $T_{\max}(G)$ be the minimum and maximum transmissions of $G$, respectively.
As in \cite{AP}, we have the following bounds. For completeness, we include a proof here.

\begin{Theorem} Let $G$ be a connected graph and $u$ and $v$ be vertices such that $T_G(u)=T_{\min}(G)$ and $T_G(v)=T_{\max}(G)$. Let $m_1=\max\{T_G(w)-(1-\alpha)d(u,w): w\in V(G)\setminus \{u\}\}$, $m_2=\min\{T_G(w)-(1-\alpha)d(v,w): w\in V(G)\setminus \{v\}\}$, and  $e(w)=\max\{d(w,z): z\in V(G)\}$ for $w\in V(G)$. Then
\begin{eqnarray*}
&&\frac{m_2+\alpha T_{\max}(G)+\sqrt{(m_2-\alpha T_{\max}(G))^2+4(1-\alpha)^2T_{\max}(G)}}{2}\\
&\le& \mu_{\alpha}(G)\\
&\le& \frac{m_1+\alpha T_{\min}(G)+\sqrt{(m_1-\alpha T_{\min}(G))^2+4(1-\alpha)^2e(u)T_{\min}(G)}}{2}.
\end{eqnarray*}
The first equality holds if and only if $G$ is a complete graph and the second equality holds if and only if $G$ is a DVDR  graph.
\end{Theorem}

\begin{proof}  Let $M$ be the submatrix of $D_{\alpha}(G)$ obtained by deleting the row and column corresponding to vertex $v$.  Let $M'$ be the matrix obtained from $M$ by reducing some  non-diagonal entries of each row with row  sum  greater than $m_2$ in $M$ such that $M'$ is nonnegative and each row sum in $M'$ is $m_2$.

Let $D^{(1)}$ be the matrix obtained from $D_{\alpha}(G)$ by replacing all $(w,v)$-entries by $1-\alpha$ for  $w\in V(G)\setminus\{v\}$, and replacing the submatrix $M$  by $M'$. Obviously,  $D_{\alpha}(G)$ and $D^{(1)}$ are nonnegative and irreducible, $D_{\alpha}(G)\ge D^{(1)}$. By Lemma~\ref{non}, $\mu_{\alpha}(G)\ge \rho (D^{(1)})$ with equality if and only if $D_{\alpha}(G)=D^{(1)}$. By applying  Lemma~\ref{ela} to
$D^{(1)}$, we obtain the lower bound for  $\mu_{\alpha}(G)$. Suppose that this lower bound is attained.
Then  $D_{\alpha}(G)=D^{(1)}$. As all $(w,v)$-entries are equal to $1-\alpha$ for  $w\in V(G)\setminus\{v\}$, implying that $d_G(v)=n-1$. As $T_G(v)=T_{\max}(G)$, $G$ is a complete graph.
Conversely, if $G$ is a complete graph,  then it is obvious that the lower bound for $\mu_{\alpha}(G)$ is attained.

Let $C$ be the submatrix of $D_{\alpha}(G)$ obtained by deleting the row and column corresponding to vertex $u$. Let $C'$ be the matrix obtained from $C$ by adding positive numbers  to non-diagonal entries of each row with row sum less than $m_1$ in $C$ such that each row sum in $C'$ is $m_1$.
Let $D^{(2)}$ be the matrix obtained from $D_{\alpha}(G)$ by replacing all $(w,u)$-entries  by $(1-\alpha)e(u)$ for $w\in V(G)\setminus\{u\}$, and replacing the submatrix $C$  by $C'$. Obviously,  $D_{\alpha}(G)$ and $D^{(2)}$ are nonnegative and irreducible, $D^{(2)}\ge D_{\alpha}(G)$. By Lemma~\ref{non}, $\mu_{\alpha}(G)\le \rho(D^{(2)})$ with equality if and only if $D_{\alpha}(G)=D^{(2)}$. By applying  Lemma~\ref{ela} to
$D^{(2)}$, we obtain the upper bound for  $\mu_{\alpha}(G)$. Suppose that this upper bound is attained.
By Lemma~\ref{non},  Then $D_{\alpha}(G)=D^{(2)}$. As all $(w,u)$-entries are equal to  $(1-\alpha)e(u)$ for $w\in V(G)\setminus\{u\}$, implying that $e(u)=1$, i.e., $d_G(u)=n-1$. Note that $T_G(w)=m_1+1-\alpha$ for all $w\in V(G)\setminus\{u\}$ and $T_{\min}(G)=T_G(u)=n-1$.  If $m_1+1-\alpha=n-1$, then $G$ is a complete graph, which is a DVDR graph. Otherwise, $m_1+1-\alpha>n-1$,
and thus  by Lemma~\ref{DVDR}, $G$ is a DVDR graph.
 Conversely, if $G$ is a DVDR graph, then  $G$ is either complete or non-complete, and by Lemma \ref{DVDR} when $G$ is non-complete and the above argument,  it is obvious that the upper bound for  $\mu_{\alpha}(G)$ is attained.
\end{proof}

We mention that more bounds  for $\mu_{\alpha}(G)$ may be derived from some known bounds for nonnegative matrices, see, e.g., \cite{DZ}.

Let $G$ be a connected graph on $n$ vertices.  As $\mu_{\alpha}(G)\le T_{\max}(G)$ with equality if and only if $G$ is transmission regular.
Recently, Liu et al.~\cite{LSX} showed that
\[
T_{\max}(G)-\mu_{0}(G)>\frac{nT_{\max}(G)-2\sigma (G)}{(nT_{\max}(G)-2\sigma (G)+1)n}
\]
and
\[
T_{\max}(G)-\mu_{\frac{1}{2}}(G)>\frac{nT_{\max}(G)-2\sigma (G)}{(2(nT_{\max}(G)-2\sigma (G))+1)n}.
\]

\begin{Theorem}\label{new1} Let $G$ be a connected non-transmission-regular graph of order $n$.
Then
\[
T_{\max}(G)-\mu_{\alpha}(G)>\frac{(1-\alpha)nT_{\max}(G)(nT_{\max}(G)-2\sigma (G))}{(1-\alpha)n^2T_{\max}(G)+4\sigma (G)(nT_{\max}(G)-2\sigma (G))}.\]
\end{Theorem}

\begin{proof}  Let  $x$ be the  $\alpha$-Perron vector  of $G$. Denote by $x_u=\max\{x_w: w\in V(G)\}$ and $x_v=\min\{x_w: w\in V(G)\}$. Since $G$ is not transmission regular, we have $x_u> x_v$, and thus
\begin{eqnarray*}
\mu_{\alpha}(G)&=&x^\top D_{\alpha}(G)x\\
&=&\alpha \sum_{w\in V(G)}T_G(w)x_w^{2}+2(1-\alpha)\sum_{\{w,z\}\subseteq V(G)}d_{wz}x_wx_z\\
&<& 2\alpha \sigma (G)x_u^2+2(1-\alpha)\sigma (G)x_u^2,
\end{eqnarray*}
implying  that $x_u^2>\frac{\mu_{\alpha}(G)}{2\sigma (G)}$.
Note that
\begin{eqnarray*}
&&T_{\max}(G)-\mu_{\alpha}(G)\\
&=&T_{\max}(G)-\alpha \sum_{w\in V(G)}T_G(w)x_w^{2}-2(1-\alpha)\sum_{\{w,z\}\subseteq V(G)}d_{wz}x_wx_z\\
&=& \sum_{w\in V(G)}(T_{\max}(G)-T_G(w))x_w^2+(1-\alpha)\sum_{\{w,z\}\subseteq V(G)}d_{wz}(x_w-x_z)^2\\
&\ge& \sum_{w\in V(G)}(T_{\max}(G)-T_G(w))x_v^2+(1-\alpha)\sum_{\{w,z\}\subseteq V(G)}d_{wz}(x_w-x_z)^2\\
&=& (nT_{\max}(G)-2\sigma (G))x_v^2+(1-\alpha)\sum_{\{w,z\}\subseteq V(G)}d_{wz}(x_w-x_z)^2.
\end{eqnarray*}

We need to estimate $\sum_{\{w,z\}\subseteq V(G)}d_{wz}(x_w-x_z)^2$. Obviously, 
\[
\sum_{\{w,z\}\subseteq V(G)}d_{wz}(x_w-x_z)^2\ge N_1+N_2,
\]
where $N_1=\sum_{w\in V(G)\setminus V(P)}\sum_{z\in V(P)}d_{wz}(x_w-x_z)^2$ and $N_2=\sum_{\{w,z\}\subseteq V(P)}d_{wz}(x_w-x_z)^2$.
Let $P=w_0, w_1,\dots, w_\ell$ be the shortest path connecting $u$ and $v$,  where $w_0=u$, $w_{\ell}=v$, and $\ell\ge 1$.
For $w\in V(G)\setminus V(P)$, by Cauchy-Schwarz inequality, we have
\[
d_{wu}(x_w-x_u)^2+d_{wv}(x_w-x_v)^2\ge (x_w-x_u)^2+(x_w-x_v)^2\ge \frac{1}{2}(x_u-x_v)^2,
\]
and thus
\begin{eqnarray*}
N_1&\ge& \sum_{w\in V(G)\setminus V(P)} \left(d_{wu}(x_w-x_u)^2+d_{wv}(x_w-x_v)^2\right)\\
&\ge & \sum_{w\in V(G)\setminus V(P)} \frac{1}{2}(x_u-x_v)^2\\
&= &\frac{n-\ell-1}{2}(x_u-x_v)^2.
\end{eqnarray*}
For $1\le i \le \ell-1$, by Cauchy-Schwarz inequality, we have
\begin{eqnarray*}
&& d_{w_0w_i}(x_{w_0}-x_{w_i})^2+d_{w_i w_\ell}(x_{w_i}-x_{w_\ell})^2\\
&\ge & \min\{i,\ell-i\}\left((x_{w_0}-x_{w_i})^2+(x_{w_i}-x_{w_\ell})^2\right)\\
&\ge& \min\{i,\ell-i\}\cdot \frac{1}{2}(x_{w_0}-x_{w_\ell})^2\\
&=& \frac{1}{2}\min\{i,\ell-i\}(x_u-x_v)^2,
\end{eqnarray*}
and thus
\begin{eqnarray*}
N_2&\ge & d_{uv}(x_u-x_v)^2+\sum_{i=1}^{\ell-1}\left(d_{w_0w_i}(x_{w_i}-x_{w_0})^2+d_{w_i w_\ell}(x_{w_i}-x_{w_\ell})^2\right)\\
&\ge & \ell (x_{u}-x_{v})^2+\sum_{i=1}^{\ell-1}\frac{1}{2}\min\{i,\ell-i\}(x_{u}-x_{v})^2\\
&=& \left(\ell+\frac{1}{2}\sum_{i=1}^{\ell-1}\min\{i,\ell-i\}\right)(x_{u}-x_{v})^2\\
&=&\begin{cases}
\frac{\ell^2+8\ell}{8}(x_u-x_v)^2 & \text{ if $\ell$ is even,}\\[2mm]
\frac{\ell^2+8\ell-1}{8}(x_u-x_v)^2 & \text{ if $\ell$ is even.}
\end{cases}
\end{eqnarray*}

\noindent{\bf Case 1.} $u$ and $v$ are adjacent, i.e., $\ell=1$.

 In this case, we have
\begin{eqnarray*}
\sum_{\{w,z\}\subseteq V(G)}d_{wz}(x_w-x_z)^2 &\geq& N_1+N_2\\
&\ge & \frac{n-1-1}{2}(x_u-x_v)^2+(x_u-x_v)^2\\
                                              &=& \frac{n}{2}(x_u-x_v)^2.
\end{eqnarray*}
Thus
\begin{eqnarray*}
T_{\max}(G)-\mu_{\alpha}(G)&\ge& (nT_{\max}(G)-2\sigma (G))x_v^2+(1-\alpha)\sum_{\{w,z\}\subseteq V(G)}d_{wz}(x_w-x_z)^2\\
&\ge & (nT_{\max}(G)-2\sigma (G))x_v^2+(1-\alpha)\frac{n}{2}(x_u-x_v)^2.
\end{eqnarray*}
Viewed as a function of $x_v$, $(nT_{\max}(G)-2\sigma (G))x_v^2+(1-\alpha)\frac{n}{2}(x_u-x_v)^2$ achieves its minimum value $\frac{(1-\alpha)n(nT_{\max}(G)-2\sigma (G))}{(1-\alpha)n+2(nT_{\max}(G)-2\sigma (G))}x_u^2$.
Recall that $x_u^2>\frac{\mu_{\alpha}(G)}{2\sigma (G)}$. Then we have
\begin{eqnarray*}
T_{\max}(G)-\mu_{\alpha}(G) 
&>&\frac{(1-\alpha)n(nT_{\max}(G)-2\sigma (G))}{(1-\alpha)n+2(nT_{\max}(G)-2\sigma (G))}
\cdot \frac{\mu_{\alpha}(G)}{2\sigma (G)}\\
&=&\frac{(1-\alpha)n(nT_{\max}(G)-2\sigma (G))T_{\max}(G)}{2\sigma (G)((1-\alpha)n+2(nT_{\max}(G)-2\sigma (G)))}\\
&&-\frac{(1-\alpha)n(nT_{\max}(G)-2\sigma (G))(T_{\max}(G)-\mu_{\alpha}(G))}{2\sigma (G)((1-\alpha)n+2(nT_{\max}(G)-2\sigma (G)))},
\end{eqnarray*}
which implies that
\[
T_{\max}(G)-\mu_{\alpha}(G)>\frac{(1-\alpha)nT_{\max}(G)(nT_{\max}(G)-2\sigma (G))}{(1-\alpha)n^2T_{\max}(G)+4\sigma (G)(nT_{\max}(G)-2\sigma (G))}.
\]

\noindent{\bf Case 2.} $u$ and $v$ are not adjacent,  i.e., $\ell\ge 2$.

Suppose first that $\ell$ is even.
Then
\begin{eqnarray*}
\sum_{\{w,z\}\subseteq V(G)}d_{wz}(x_w-x_z)^2 &\geq&N_1+N_2\\
&\ge & \frac{n-\ell-1}{2}(x_u-x_v)^2+\frac{\ell^2+8\ell}{8}(x_u-x_v)^2\\
                                              &=& \frac{\ell ^2+4\ell +4n-4}{8}(x_u-x_v)^2.
\end{eqnarray*}
Thus
\begin{eqnarray*}
T_{\max}(G)-\mu_{\alpha}(G)&\ge& (nT_{\max}(G)-2\sigma (G))x_v^2+(1-\alpha)\sum_{\{w,z\}\subseteq V(G)}d_{wz}(x_w-x_z)^2\\
&\ge& (nT_{\max}(G)-2\sigma (G))x_v^2+(1-\alpha)\frac{\ell ^2+4\ell +4n-4}{8}(x_u-x_v)^2.
\end{eqnarray*}
Viewed as a function of $x_v$, $(nT_{\max}(G)-2\sigma (G))x_v^2+(1-\alpha)\frac{\ell ^2+4\ell +4n-4}{8}(x_u-x_v)^2$ achieves its minimum value $\frac{(1-\alpha)(nT_{\max}(G)-2\sigma (G))(\ell^2+4\ell+4n-4)}{(1-\alpha)(\ell^2+4\ell+4n-4)+8(nT_{\max}(G)-2\sigma (G))}x_u^2$. As $x_u^2>\frac{\mu_{\alpha}(G)}{2\sigma (G)}$, we have
\[
T_{\max}(G)-\mu_{\alpha}(G)
> \frac{(1-\alpha)(nT_{\max}(G)-2\sigma (G))(\ell^2+4\ell+4n-4)}{(1-\alpha)(\ell^2+4\ell+4n-4)+8(nT_{\max}(G)-2\sigma (G))}\cdot \frac{\mu_{\alpha}(G)}{2\sigma (G)},
\]
i.e.,
\[
T_{\max}(G)-\mu_{\alpha}(G)>\frac{(1-\alpha)(nT_{\max}(G)-2\sigma (G))(\ell^2+4\ell+4n-4)T_{\max}(G)}{(1-\alpha)(\ell^2+4\ell+4n-4)nT_{\max}(G)+16\sigma (G)(nT_{\max}(G)-2\sigma (G))}.
\]
As a function of $\ell$, the expression in the right hand side in the above inequality is strictly increasing for $\ell\ge 2$. Thus we have
\begin{eqnarray*}
T_{\max}(G)-\mu_{\alpha}(G)&>&\frac{(1-\alpha)(nT_{\max}(G)-2\sigma (G))(n+2)T_{\max}(G)}{(1-\alpha)(n+2)nT_{\max}(G)+4\sigma (G)(nT_{\max}(G)-2\sigma (G))}\\
&>&\frac{(1-\alpha)nT_{\max}(G)(nT_{\max}(G)-2\sigma (G))}{(1-\alpha)n^2T_{\max}(G)+4\sigma (G)(nT_{\max}(G)-2\sigma (G))}.
\end{eqnarray*}

Now suppose that $\ell$ is odd. 

Then
\begin{eqnarray*}
\sum_{\{w,z\}\subseteq V(G)}d_{wz}(x_w-x_z)^2 &\geq& N_1+N_2\\
&\ge & \frac{n-\ell-1}{2}(x_u-x_v)^2+\frac{\ell^2+8\ell-1}{8}(x_u-x_v)^2\\
&=& \frac{\ell ^2+4\ell +4n-5}{8}(x_u-x_v)^2.
\end{eqnarray*}
Thus, as early, we have
\begin{eqnarray*}
&&T_{\max}(G)-\mu_{\alpha}(G)\\
&\ge& (nT_{\max}(G)-2\sigma (G))x_v^2+(1-\alpha)\frac{\ell ^2+4\ell +4n-5}{8}(x_u-x_v)^2\\
&\ge& \frac{(1-\alpha)(\ell^2+4\ell+4n-5)(nT_{\max}(G)-2\sigma (G))}{(1-\alpha)(\ell^2+4\ell+4n-5)+8(nT_{\max}(G)-2\sigma (G))}x_u^2\\
&>& \frac{(1-\alpha)(\ell^2+4\ell+4n-5)(nT_{\max}(G)-2\sigma (G))}{(1-\alpha)(\ell^2+4\ell+4n-5)+8(nT_{\max}(G)-2\sigma (G))}\cdot \frac{\mu_{\alpha}(G)}{2\sigma (G)},
\end{eqnarray*}
implying that
\[
T_{\max}(G)-\mu_{\alpha}(G)>\frac{(1-\alpha)(nT_{\max}(G)-2\sigma (G))(\ell^2+4\ell+4n-5)T_{\max}(G)}{(1-\alpha)(\ell^2+4\ell+4n-5)nT_{\max}(G)+16\sigma (G)(nT_{\max}(G)-2\sigma (G))}.
\]
As a function of $\ell$, the expression in the right hand side in the above inequality is strictly increasing for $\ell\ge 3$. Thus we have
\begin{eqnarray*}
T_{\max}(G)-\mu_{\alpha}(G)&>&\frac{(1-\alpha)(nT_{\max}(G)-2\sigma (G))(4+n)T_{\max}(G)}{(1-\alpha)(4+n)nT_{\max}(G)+4\sigma (G)(nT_{\max}(G)-2\sigma (G))}\\
&>&\frac{(1-\alpha)nT_{\max}(G)(nT_{\max}(G)-2\sigma (G))}{(1-\alpha)n^2T_{\max}(G)+4\sigma (G)(nT_{\max}(G)-2\sigma (G))}.
\end{eqnarray*}

The result follows by combining Cases 1 and 2.
\end{proof}

If $\alpha=0, \frac{1}{2}$, then the bound for $T_{\max}(G)-\rho_{0}(G)$ in  Theorem~\ref{new1} reduces to
\[
T_{\max}(G)-\mu_{0}(G)>\frac{(nT_{\max}(G)-2\sigma (G))nT_{\max}(G)}{n^2T_{\max}(G)+4(nT_{\max}(G)-2\sigma (G))\sigma (G)} 
\]
and
\[
T_{\max}(G)-\mu_{\frac{1}{2}}(G)>\frac{(nT_{\max}(G)-2\sigma (G))nT_{\max}(G)}{n^2T_{\max}(G)+8(nT_{\max}(G)-2\sigma (G))\sigma (G)}. 
\]

\section{Effect of graft transformations on distance $\alpha$-spectral radius}

In this section, we study the effect of some graft transformations on distance $\alpha$-spectral radius.

A path $u_0\dots u_r$ (with $r\geq 1$) in a graph $G$ is called a pendant path (of length $r$)
at $u_0$ if $d_G(u_0)\geq 3$, the degrees of $u_1,\dots,u_{r-1}$ (if any exists) are all equal to $2$ in $G$, and $d_G(u_r)=1$.
A pendant path of length $1$ at $u_0$ is called a pendant edge at $u_0$.

A vertex of a graph is a pendant vertex if its degree is $1$.
The neighbor of the pendant vertex in a graph is called a quasi-pendant vertex. A non-pendant edge in a graph is an edge such that both end vertices are not pendant vertices.
A cut edge of a connected graph is an edge whose removal yields a disconnected graph.

If $P$ is a pendant path of $G$ at $u$ with length $r\geq1$, then we say $G$ is obtained from $H$ by attaching a pendant path $P$ of length $r$ at $u$ with $H=G[V(G)\setminus(V(P)\setminus\{u\})]$.
If the pendant path of length $1$ is attached to a vertex $u$ of $H$, then we also say that a pendant vertex  is attached to $u$.

\begin{Theorem}\label{new1} Let $G$ be a connected graph and $uv$ a non-pendant cut edge of $G$. Let $G_{uv}$ be the graph obtained from $G$ by
identifying vertices $u$ and $v$  to vertex $v$ and attaching a pendant vertex $u$ to $v$. If at least one of $\{u,v\}$ is a quasi-pendant vertex in $G$, then $\mu_{\alpha}(G)>\mu_{\alpha}(G_{uv})$. 
\end{Theorem}

\begin{proof}  Assume that $v$ is a quasi-pendant vertex in $G$. Let  $v'$ be a pendant neighbor of $v$, and
let $G_1$ and $G_2$ be the components  of $G-uv$ containing $u$ and $v$, respectively,  see Fig.~1.

\begin{center}
\begin{picture}(310,80)
\put(50,35){\circle* {3.5}} \put(70,35){\circle* {3.5}}\put(50,35){\line(1,0) {20}}\put(70,35){\line(1,0) {15}}\put(85,35){\circle* {3.5}}
\put(50,35){\line(-2,1) {30}} \put(50,35){\line(-2,-1) {30}}
\put(70,35){\line(2,1) {30}} \put(70,35){\line(2,-1) {30}}
\bezier{600}(20,20)(-2,35)(20,50) \bezier{600}(100,20)(122,35)(100,50)
\put(48,25){$u$}\put(88,35){$v'$}
\put(-10,30){$G_1$} \put(116,30){$G_2$}
\put(66,25){$v$} \put(55,2){$G$}

\put(250,35){\circle* {3.5}} \put(250,55){\circle* {3.5}}\put(250,35){\line(0,1) {20}}\put(250,35){\line(1,0) {15}}\put(265,35){\circle* {3.5}}
\put(250,35){\line(-2,1) {30}} \put(250,35){\line(-2,-1) {30}}
\put(250,35){\line(2,1) {30}} \put(250,35){\line(2,-1) {30}}
\bezier{600}(220,20)(198,35)(220,50) \bezier{600}(280,20)(302,35)(280,50)
\put(248,25){$v$}\put(268,35){$v'$}
\put(190,30){$G_1$} \put(296,30){$G_2$}
\put(249,59){$u$} \put(246,2){$G_{uv}$}

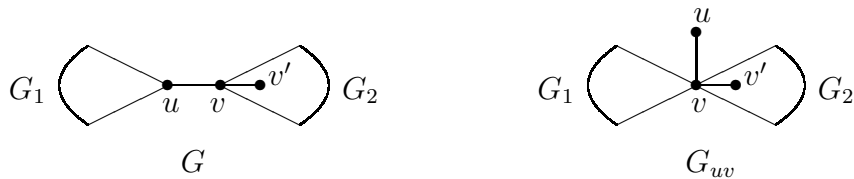
\put(40,-20) {Fig.~1. The graphs $G$ and $G_{uv}$ in Theorem~\ref{new1}. }
\end{picture}
\end{center}
\vspace{10mm}

Let $x$ be the distance $\alpha$-Perron vector  of $G_{uv}$. By Lemma~\ref{auto},
$x_u=x_{v'}$.

As we pass from $G$ to $G_{uv}$, the distance between a vertex in $V(G_1)\setminus\{u\}$ and a vertex in $V(G_2)$ is decreased by $1$, the distance between a vertex $V(G_1)\setminus\{u\}$ and $u$ is increased by $1$, and the distances between all other vertex pairs remain unchanged.
Thus
\begin{eqnarray*}
&&\mu_{\alpha} (G)-\mu_{\alpha} (G_{uv})\\
&\geq &  x^\top(D_{\alpha}(G)-D_{\alpha}(G_{uv}))x\\
&=&\sum_{w\in V(G_1)\setminus\{u\}}\sum_{z\in V(G_2) }\left(\alpha\left(x_w^2+x_z^2\right)+2(1-\alpha)x_{w}x_{z}\right)\\
&&-\sum_{ w\in V(G_1)\setminus\{u\}}\left(\alpha\left(x_w^2+x_u^2\right)+2(1-\alpha)x_{w}x_{u}\right)\\
&\geq&\sum_{w\in V(G_1)\setminus\{u\}}\left(\alpha\left(x_w^2+x_v^2\right)+2(1-\alpha)x_{w}x_{v}+\alpha\left(x_{w}^2+x_{v'}^2\right)+2(1-\alpha)x_{w}x_{v'}\right)\\
&&-\sum_{ w\in V(G_1)\setminus\{u\}}\left(\alpha\left(x_w^2+x_u^2\right)+2(1-\alpha)x_{w}x_{u} \right)\\
&=&\sum_{w\in V(G_1)\setminus\{u\}}\left(\alpha\left(x_w^2+x_v^2\right)+2(1-\alpha)x_{w}x_{v}\right)\\
&>&0.
\end{eqnarray*}
Therefore $\mu_{\alpha} (G)-\mu_{\alpha} (G_{uv})>0$, i.e., $\mu_{\alpha}(G)>\mu_{\alpha}(G_{uv})$.
\end{proof}

The previous theorem has been established for $\alpha=0,\frac{1}{2}$ in \cite{WZ,LZ}.

\begin{Theorem}\label{333} Let $G$ be a connected graph with $k$ edge--disjoint nontrivial induced subgraphs  $G_1, \dots, G_k$  such that  $V(G_i)\cap V(G_j)=\{u\}$ for $1\le i<j\le k$ and $\cup_{i=1}^kV(G_i)=V(G)$, where $k\ge 3$. Let $K$ be a nonempty subset of $\{3, \dots, k\}$ and let
$N_K=\cup_{i\in K}N_{G_i}(u)$.
 For   $v'\in V(G_1)\setminus\{u\}$ and
$v''\in V(G_2)\setminus \{u\}$, let
\[
G' =G-\{uw:w\in N_K\}+\{v'w: w\in N_K\}
\]
and
\[
G'' =G-\{uw:w\in N_K\}+\{v''w: w\in N_K\}.
\]
 Then  $\mu_{\alpha}(G)< \mu_{\alpha}(G')$ or $\mu_{\alpha}(G)<\mu_{\alpha}(G'')$.
\end{Theorem}

\begin{proof}  Let $x$ be the distance $\alpha$-Perron vector  of $G$. Let $V_K=\left(\cup_{i\in K} V(G_i)\right)\setminus\{u\}$. Let
\begin{eqnarray*} \Gamma&=&\sum_{w\in V(G_2)\setminus\{u\}}\sum_{z\in V_K}\left(\alpha\left(x_w^2+x_z^2\right)+2(1-\alpha)x_{w}x_{z}\right)\\
&&-\sum_{w\in V(G_1)\setminus\{u\}}\sum_{z\in V_K}\left(\alpha\left(x_w^2+x_z^2\right)+2(1-\alpha)x_{w}x_{z}\right).
\end{eqnarray*}

 As we pass from $G$ to $G'$, the distance between a vertex in $V(G_2)$ and a vertex in $V_K$ is increased by $d_G(u,v')$, the distance between a vertex $w$ in $V(G_1)\setminus\{u\}$ and a vertex in $V_K$ is decreased by $d_G(w,u)-d_G(w,v')$, which is at most $d_G(u,v')$, and the distances between all other vertex pairs are increased or remain unchanged. Thus
\begin{eqnarray*}
&&\mu_{\alpha} (G')-\mu_{\alpha} (G)\\
&\geq & x^\top (D_{\alpha}(G')-D_{\alpha}(G))x \\
&\geq& \sum_{w\in V(G_2)}\sum_{z\in V_K} \left(d_{G}(u,v')\left(\alpha\left(x_w^2+x_z^2\right)+2(1-\alpha)x_{w}x_{z}\right)\right)\\
&&-\sum_{w\in V(G_1)\setminus\{u\}}\sum_{z\in V_K} \left(d_{G}(u,v')\left(\alpha\left(x_w^2+x_z^2\right)+2(1-\alpha)x_{w}x_{z}\right)\right)\\
&= & d_{G}(u,v')\left(\Gamma +\sum_{z\in V_K} \left(\alpha\left(x_u^2+x_z^2\right)+2(1-\alpha)x_{u}x_z\right)\right)\\
&>& d_G(u,v')\Gamma.
\end{eqnarray*}
If $\Gamma\ge 0$, then $\mu_{\alpha} (G')-\mu_{\alpha} (G)> d_G(u,v')\Gamma\ge 0$, implying that $\mu_{\alpha} (G)<\mu_{\alpha} (G')$. Suppose that $\Gamma<0$.
As we pass from $G$ to $G''$, the distance between a vertex in $V(G_1)$ and a vertex in $V_K$ is increased by $d_G(u,v'')$, the distance between a vertex $w$ in $V(G_2)\setminus\{u\}$ and a vertex in $V_K$ is decreased by $d_G(w,u)-d_G(w,v'')$, which is at most $d_G(u,v'')$, and the distances between all other vertex pairs are increased or remain unchanged. Thus
\begin{eqnarray*}
&&\mu_{\alpha} (G'')-\mu_{\alpha} (G)\\
&\geq & x^\top (D_{\alpha}(G'')-D_{\alpha}(G))x\\
&\geq& \sum_{w\in V(G_1)}\sum_{z\in V_K} \left(d_{G}(u,v'')\left(\alpha\left(x_w^2+x_z^2\right)+2(1-\alpha)x_{w}x_{z}\right)\right)\\
&&-\sum_{w\in V(G_2)\setminus\{u\}}\sum_{z\in V_K} \left(d_{G}(u,v'')\left(\alpha\left(x_w^2+x_z^2\right)+2(1-\alpha)x_{w}x_{z}\right)\right)\\
&= & d_{G}(u,v'')\left(-\Gamma+ \sum_{z\in V_K} \left(\alpha\left(x_u^2+x_z^2\right)+2(1-\alpha)x_{u}x_{z}\right)\right)\\
&>& d_G(u,v'')(-\Gamma )\\
&>& 0,
\end{eqnarray*}
implying that   $\mu_{\alpha} (G'')-\mu_{\alpha} (G)> 0$, i.e., $\mu_{\alpha}(G)<\mu_{\alpha}(G'')$.
\end{proof}

Weak versions of previous theorem for $\alpha=0$ have been given in \cite{Yu,XZD} and a weak version
 for $\alpha=\frac{1}{2}$ may be found  in \cite{LZ}.

For positive integer $p$ and a graph $G$ with $u\in V(G)$, let $G(u;p)$  be the graph obtained from $G$ by attaching a pendant path of
length $p$ at $u$. Let   $G(u;0)=G$, and in this case a pendant path of length $0$ is understood the trivial path consisting of a single vertex $u$.

For nonnegative integers $p$, $q$ and a graph $G$, let $G_u(p,q)$ or simply $G_{p,q}$ be the graph $H(u;q)$ with $H=G(u;p)$.

The following corollary has been given for $\alpha=0$ in \cite{SI,XZD} and $\alpha=\frac{1}{2}$ in~\cite{LL,LZ}.

\begin{Corollary}\label{pq} Let $H$ be a nontrivial connected graph with $u\in V(H)$. If $p\ge q\ge 1$, then $\mu_{\alpha} (H_u(p,q))<\mu_{\alpha} (H_u(p+1,q-1))$.
\end{Corollary}

\begin{proof}
Let $G=H_u(p,q)$.  Let $P=uu_1\dots u_p$ and $Q=uv_1\dots v_q$ be two pendant paths of lengths $p$ and $q$, respectively in $G$.
Using the notations in Theorem~\ref{333} with $k=3$,  $G_1=P$, $G_2=Q$, $G_3=H$, $v'=u_{p-q+1}$ and $v''=v_1$, we have $G'\cong G'' \cong H_u(p+1,q-1)$, and thus
by Theorem~\ref{333},
we have
 $\mu_{\alpha} (H_u(p,q))<\mu_{\alpha} (H_u(p+1,q-1))$.
\end{proof}

\begin{Theorem} \label{444} Let $G$ be a connected graph with three edge--disjoint induced subgraphs  $G_1$, $G_2$ and $G_3$ such that $V(G_1)\cap V(G_3)=\{u\}$, $V(G_2)\cap V(G_3)=\{v\}$, $\cup_{i=1}^3V(G_i)=V(G)$, and $G_1-u$,  $G_2-v$, and $G_3-u-v$ are all nontrivial.
Suppose that $uv\in E(G_3)$.  
For $u'\in N_{G_1}(u) $ and $v'\in N_{G_2}(v)$, let
\[
G'=H+\{u'w: w\in N_{G_3-uv}(u)\}+\{uw: w\in N_{G_3-uv}(v)\}\]
and
\[
G''=H+\{vw: w\in N_{G_3-uv}(u)\}+\{v'w: w\in N_{G_3-uv}(v)\},\]
where $H=G-\{uw: w\in N_{G_3-uv}(u)\}-\{vw: w\in N_{G_3-uv}(v)\}$.
Then $\mu_{\alpha} (G)<\mu_{\alpha} (G')$ or $\mu_{\alpha} (G)<\mu_{\alpha} (G'')$.
\end{Theorem}

\begin{proof}
Let $x$ be the distance $\alpha$-Perron vector  of $G$. Let
\begin{eqnarray*}
\Gamma&=&\sum_{w\in V(G_2)}\sum_{z\in V(G_3)\setminus\{u,v\} }\left(\alpha\left(x_w^2+x_z^2\right)+2(1-\alpha)x_{w}x_{z}\right)\\
&&- \sum_{w\in V(G_1)}\sum_{z\in V(G_3)\setminus\{u,v\}}\left(\alpha\left(x_w^2+x_z^2\right)+2(1-\alpha)x_{w}x_{z}\right).
\end{eqnarray*}

As we pass from $G$ to $G'$, the distance between a vertex in $V(G_2)$ and a vertex in $V(G_3)\setminus\{u,v\}$ is increased by $1$, the distance between a vertex in $V(G_1)$ and a vertex in $V(G_3)\setminus\{u,v\}$ may be increased, unchanged, or  decreased by  $1$, and the distances between any other vertex pairs remain unchanged. Thus
\begin{eqnarray*}
\mu_{\alpha} (G')-\mu_{\alpha} (G)
&\ge &  x^\top (D_{\alpha}(G')- D_{\alpha}(G))x\\
&\ge&\sum_{w\in V(G_2)}\sum_{z\in V(G_3)\setminus\{u,v\} }\left(\alpha\left(x_w^2+x_z^2\right)+2(1-\alpha)x_{w}x_{z}\right)\\
&&-\sum_{w\in V(G_1)}\sum_{z\in V(G_3)\setminus\{u,v\} }\left(\alpha\left(x_w^2+x_z^2\right)+2(1-\alpha)x_{w}x_{z}\right)\\
&=&\Gamma.
\end{eqnarray*}
If $\Gamma\ge 0$, then
$\mu_{\alpha} (G')-\mu_{\alpha} (G)\ge 0$,
i.e., $\mu_{\alpha} (G)\le \mu_{\alpha} (G')$. If $\mu_{\alpha} (G)=\mu_{\alpha} (G')$, then $\mu_{\alpha} (G')=x^{\top}D_{\alpha}(G')x$, implying that $x$ is the distance
$\alpha$-Perron vector of $G'$. By the $\alpha$-eigenequations of $G$ and $G'$ at $v$, we have
\begin{eqnarray*}
0&=& \mu_{\alpha}(G')x_v-\mu_{\alpha}(G)x_v\\
&=&\sum_{w\in V(G_3)\setminus\{u,v\}} \left(d_{G'}(v,w)-d_G(v,w)\right)(\alpha x_v+(1-\alpha)x_w)\\
&=&\sum_{w\in V(G_3)\setminus\{u,v\}} (\alpha x_v+(1-\alpha)x_w)\\
&>& 0,
\end{eqnarray*}
a contradiction. Thus, if $\Gamma\ge 0$, then $\mu_{\alpha} (G)<\mu_{\alpha} (G')$.

Suppose that $\Gamma<0$. As earlier, we have
\begin{eqnarray*}
\mu_{\alpha}(G'')-\mu_{\alpha}(G)
&\ge &  x^\top (D_{\alpha}(G'')-D_{\alpha}(G))x \\
&\ge &\sum_{w\in V(G_1)}\sum_{z\in V(G_3)\setminus\{u,v\} }\left(\alpha\left(x_w^2+x_z^2\right)+2(1-\alpha)x_{w}x_{z}\right)\\
&&-\sum_{w\in V(G_2)}\sum_{z\in V(G_3)\setminus\{u,v\} }\left(\alpha\left(x_w^2+x_z^2\right)+2(1-\alpha)x_{w}x_{z}\right)\\
&=& -\Gamma\\
&>&0,
\end{eqnarray*}
and thus $\mu_{\alpha}(G)<\mu_{\alpha}(G'')$.
\end{proof}

A weak version of previous theorem for $\alpha=\frac{1}{2}$ has been established in \cite{LZ}.

For nonnegative integers $p$, $q$ and a graph $G$ with $u,v\in V(G)$, let $G_{u,v}(p,q)$  be the graph $H(v;q)$ with $H=G(u;p)$.

Similar versions for the following corollary have been given for $\alpha=0,\frac{1}{2}$ in \cite{Go,LL}.

\begin{Corollary}\label{adpq} Let $H$ be a connected graph of order at least $3$ with $uv\in E(H)$. Suppose that $\eta(u)=v$
for some automorphism $\eta$  of $G$.
For $p\ge q\ge 1$, 
we have
 $\mu_{\alpha}(H_{u,v}(p,q))<\mu_{\alpha}(H_{u,v}(p+1,q-1))$.
\end{Corollary}

\begin{proof}  Let $G=H_{u,v}(p,q)$.  Let $P=uu_1\dots u_p$ and $Q=vv_1\dots v_q$ be two pendant paths of lengths $p$ and $q$ in $G$ at $u$ and $v$, respectively.
Using the notations of Theorem~\ref{444} with $G_1=P$, $G_2=Q$, $G_3=H$, $u'=u_1$ and $v'=v_1$,  we have $G'\cong H_{u,v}(p-1,q+1)$ and $G''\cong H_{u,v}(p+1,q-1)$, and thus
by Theorem~\ref{444},
we have
\begin{eqnarray}\label{Co2}
\mu_{\alpha} (H_{u,v}(p,q))<\max\{\mu_{\alpha}(H_{u,v}(p-1,q+1)), \mu_{\alpha}(H_{u,v}(p+1,q-1))\}.
\end{eqnarray}

If $p=q$ ($p=q+1$, respectively), then $H_{u,v}(p-1,q+1)\cong H_{u,v}(p+1,q-1)$ ($H_{u,v}(p,q)\cong H_{u,v}(p-1,q+1)$, respectively)
as  $\eta(u)=v$
for some automorphism $\eta$  of $G$,
and thus from \eqref{Co2}, we have
$\mu_{\alpha} (G)<\mu_{\alpha}(H_{u,v}(p+1,q-1))$.
Suppose that $p\ge q+2$ and  $\mu_{\alpha} (G)<\mu_{\alpha}(H_{u,v}(p-1,q+1))$.

 If $p\not\equiv q \pmod 2$, then by using (\ref{Co2}) repeatedly, we have
\begin{eqnarray*}
\mu_{\alpha} (G)&\le&\mu_{\alpha}\left(H_{u,v}\left(\frac{p+q+3}{2},\frac{p+q-3}{2}\right)\right)\\
&<&\mu_{\alpha}\left(H_{u,v}\left(\frac{p+q+1}{2},\frac{p+q-1}{2}\right)\right)\\
&<& \mu_{\alpha} \left(H_{u,v}\left(\frac{p+q+3}{2},\frac{p+q-3}{2}\right)\right),
\end{eqnarray*}
which is impossible. If $p\equiv q \pmod 2$, then by using (\ref{Co2}) repeatedly, we have
\begin{eqnarray*}
\mu_{\alpha} (G)&\le& \mu_{\alpha}\left(H_{u,v}\left(\frac{p+q}{2}+1,\frac{p+q}{2}-1\right)\right)\\
&<&\mu_{\alpha}\left(H_{u,v}\left(\frac{p+q}{2},\frac{p+q}{2}\right)\right)\\
&<& \mu_{\alpha} \left(H_{u,v}\left(\frac{p+q}{2}-1,\frac{p+q}{2}+1\right)\right),
\end{eqnarray*}
which is also impossible. Therefore
$\mu_{\alpha} (H_{u,v}(p,q))<\mu_{\alpha} (H_{u,v}(p+1,q-1))$.
\end{proof}


\section{Graphs with small distance $\alpha$-spectral radius}

In this section, we will determine the graphs with minimum distance $\alpha$-spectral radius among  trees and unicyclic graphs.



\begin{Theorem}\label{tmax} Let $G$ be a tree of order $n\geq 4$.
Then $\mu_{\alpha}(G)\geq \mu_{\alpha}(S_n)$ with equality if and only if $G\cong S_n$.
\end{Theorem}

\begin{proof}  Let $G$ be a tree of order $n$ with minimum distance $\alpha$-spectral radius.
Let $d$ be the diameter of $G$. Obviously, $d\geq 2$. Suppose that $d\ge 3$. Let $v_0v_1\dots v_d$ be a diametral path of $G$. Note that $v_1$ is a quasi-pendant vertex in $G$. By Theorem~\ref{new1}, $\mu_{\alpha}(G_{v_1v_2})<\mu_{\alpha}(G)$, a contradiction. Thus $d= 2$, i.e., $G\cong S_n$.
\end{proof}

In Theorem~\ref{tmax}, the case  $\alpha=0$ has been known in \cite{SI} and the case  $\alpha=\frac{1}{2}$ has been known in \cite{LZ,XZL}.

For $n-1\ge 3$ and $1\le a\le \left \lfloor \frac{n-2}{2}\right \rfloor$, let $D_{n,a}$ be the tree obtained from vertex-disjoint $S_{a+1}$ with center $u$ and $S_{n-a-1}$ with center $v$ by adding an edge $uv$.
Let $T$ be a tree of order $n$ with minimum distance $\alpha$-spectral radius, where $T\ncong S_n$.
Let $d$ be the diameter of $T$. Obviously, $d\geq 3$. Suppose that $d\geq 4$. Let $v_0v_1\dots v_d$ be a diametral path of $T$. Note that $v_1$ is a quasi-pendant vertex in $T$ and $T_{v_1v_2} \ncong S_n$. By Theorem~\ref{new1}, $\mu_{\alpha}(T_{v_1v_2})<\mu_{\alpha}(T)$, a contradiction. Thus $d=3$, implying that $T\cong D_{n,a}$ for some $a$ with $1\leq a\leq \lfloor\frac{n-2}{2}\rfloor$.


%

\begin{Lemma}\label{tn}\cite{XZL} Let $G$ be a unicyclic graph of order $n\ge 6$  different from $S_n^{+}$, where $S_n^{+}$ is the graph obtained from $S_n$ by adding an edge between two vertices of degree one. Then
\[\sigma (G)\ge n^2-n-4>\sigma (S_n^{+})=n^2-2n.\]
\end{Lemma}

\begin{Theorem} \label{2m} Let $G$ be a  unicyclic graph of order  $n\ge 8$. Then $\mu_{\alpha}(G)\ge \mu_{\alpha}(S_n^{+})$ with equality if and only if $G\cong S_n^{+}$.
\end{Theorem}

\begin{proof}  Suppose that  $G\ncong S_n^{+}$. We only need to show that  $\mu_{\alpha}(G)>\mu_{\alpha}(S_n^{+})$.

By Lemmas~\ref{xz} and \ref{tn}, we have
\[
\mu_{\alpha}(G)\ge \frac{2\sigma (G)}{n}\ge \frac{2(n^2-n-4)}{n}.\]
As $\mu_{\alpha}(G)$ is bound above by the maximum row sum of $D_{\alpha}(G)$, and it is attained if and only if all row sums of $D_{\alpha}(G)$ are equal \cite[p. 24,
Theorem 1.1]{HM}. Thus
\[
\mu_{\alpha}(S_n^{+})< T_{\max}(S_n^{+})=2n-3.
\]
Since  $n\ge 8$, we have
\[
\mu_{\alpha}(G)\ge \frac{2(n^2-n-4)}{n}\ge 2n-3>\mu_{\alpha}(S_n^{+}),
\]
as desired.
\end{proof}

The result in  Theorem~\ref{2m} for  $\alpha=0, \frac{1}{2}$ has been known in \cite{Yu2,XZL}.

\section{Graphs with large distance $\alpha$-spectral radius}

In this section, we will  determine the graphs with maximum  distance $\alpha$-spectral radius among some classes of graphs.
For examples, we determine the unique connected graphs of order $n\ge 4$ with maximum and second maximum distance $\alpha$-spectral radius, respectively in Theorem~\ref{gen} and the unique graph with maximum distance $\alpha$-spectral radius among connected graphs with fixed clique number in Theorem~\ref{888}.

For $2\le \Delta\le n-1$, let  $B_{n,\Delta}$ be a tree obtained by attaching $\Delta-1$ pendant vertices to a terminal vertex of the path $P_{n-\Delta+1}$. In particular, $B_{n,2}=P_n$ and $B_{n, n-1}=S_n$. The following theorem for $\alpha=0,\frac{1}{2}$ was given in \cite{SI,LZ} for trees.

\begin{Theorem}\label{max degree}
Let $G$ be a connected graph of order $n\ge 5$  with maximum degree $\Delta$, where $2\le \Delta\le n-1$. Then $\mu_{\alpha} (G)\leq \mu_{\alpha} (B_{n,\Delta})$ with equality if and only if $G\cong B_{n,\Delta}$.
\end{Theorem}

\begin{proof}   Let $G$ be a  graph with maximum distance $\alpha$-spectral radius among connected graphs of order $n$ with maximum degree $\Delta$.
Obviously, $G$ has a spanning tree $T$ with maximum degree $\Delta$. By Lemma~\ref{ad}, $\mu_{\alpha}(G)\le \mu_{\alpha}(T)$ with equality if and only if $G\cong T$. Thus $G$ is a tree.

It is trivial if $\Delta=2, n-1$.
Suppose that $3\leq \Delta \leq n-2$. We only need to show that $G\cong B_{n,\Delta}$.

Let  $u\in V(G)$ with $d_G(u)=\Delta$.
Suppose that there exists a vertex different from $u$ with degree at least $3$.  Then we may choose such a vertex $w$ of degree at least $3$ such that $d_G(u,w)$ is as large as possible. Obviously, there are two pendant paths, say $P$ and $Q$,  at $w$ of lengths  at least $1$. Let $p$ and $q$ be the lengths of $P$ and $Q$, respectively. Assume that   $p\ge q$.
Let $H=G[V(G)\setminus((V(P)\cup V(Q))\setminus\{w\})]$. Then $G\cong H_w(p,q)$.  Obviously,  $G'=H_w(p+1,q-1)$ is a tree of order $n$ with maximum degree $\Delta$.  By Corollary~\ref{pq}, $\mu_{\alpha}(G)<\mu_{\alpha}(G')$, a contradiction. Then $u$ is the unique vertex of $G$ with degree at least $3$, and thus $G$ consists of $\Delta$ pendant paths, say $Q_1, \dots, Q_{\Delta}$ at $u$. If two of them, say $Q_i$ and $Q_j$ with $i\ne j$ are of lengths at least $2$, then $G\cong H'_u(r,s)$,  where  $H'=G[V(G)\setminus((V(Q_i)\cup V(Q_j))\setminus\{u\})]$, and $r$ and $s$ are the lengths of $Q_i$ and $Q_j$, respectively. Assume that $r\ge s$. Obviously,  $G''=H'_u(r+1,s-1)$ is a tree of order $n$ with maximum degree $\Delta$. By Corollary~\ref{pq}, $\mu_{\alpha} (G)< \mu_{\alpha} (G'')$, also a contradiction. Thus there is exactly one pendant path at $u$ of length at least $2$, implying that $G\cong B_{n,\Delta}$.
\end{proof}

%
%

If $G$ is a connected graph of order $1$ or $2$, then $G\cong P_n$. If $G$ is a connected graph of order $3$, then $G\cong P_3, K_3$, and by Lemma~\ref{ad}, $\mu_{\alpha}(K_3)<\mu_{\alpha}(P_3)$. 

Ruzieh and Powers \cite{RP}  showed that  $P_n$ is the unique connected graph of order $n$ with maximum distance $0$-spectral
radius, and it was proved in \cite{WZ} that $B_{n,3}$  is the unique tree of order $n$ different from $P_n$ with maximum distance $0$-spectral
radius. For $\alpha=\frac{1}{2}$, the following theorem was given in \cite{LZ}.

\begin{Theorem}\label{gen} Let $G$ be a connected graph of order $n\geq 4$, where  $G\not\cong P_n$.
Then $\mu_{\alpha}(G)\leq \mu_{\alpha}(B_{n,3})<\mu_{\alpha}(P_n)$ 
with equality if and only if
$G\cong B_{n,3}$.
\end{Theorem}

\begin{proof}   First suppose that $G$ is a tree.
If $n=4$, then the result follows from Theorem \ref{new1}. Suppose that $n\geq5$.
Let $\Delta$ be the maximum degree of $G$.
Since $G\ncong P_n$, we have $\Delta\geq3$. By Theorem~\ref{max degree},  $\mu_{\alpha}(G)\le \mu_{\alpha}(B_{n,\Delta})$ with equality if and only if $G\cong B_{n, \Delta}$.
By Corollary~\ref{pq}, 
$\mu_{\alpha}(G)\le\mu_{\alpha}(B_{n,\Delta})\le \mu_{\alpha}(B_{n,3})<\mu_{\alpha}(P_n)$ with equalities if and only if $\Delta=3$ and  $G\cong B_{n, \Delta}$, i.e., $G\cong B_{n,3}$.

Now suppose that $G$ is not a tree. Then $G$ contains at least one cycle. If there is a spanning tree $T$ with $T\ncong P_n$, then by Lemma~\ref{ad} and the above argument, we have $\mu_{\alpha}(G)< \mu_{\alpha}(T)\leq \mu_{\alpha}(B_{n,3})$. If any  spanning tree of $G$ is a path, then $G$ is a cycle $C_n$. Now we only need to show that $\mu_{\alpha}(C_n)<\mu_{\alpha}(B_{n,3})$.

Let  $C_n=u_1u_2\dots u_nu_1$ and $T'=C_n-u_1u_2-u_2u_3+u_2u_n$. Obviously, $T'\cong B_{n,3}$.  Let $x$ be the distance $\alpha$-Perron vector  of $C_n$. By Lemma~\ref{ad}, we have $x_{u_1}=\dots=x_{u_{n}}$.
As we pass from $C_n$ to $T'$, the distance between $u_2$ and $u_1$ is increased by $1$, the distance between $u_2$ and $u_i$ with $3\leq i\leq \left\lceil\frac{n+1}{2}\right\rceil$ is increased by $n-2i+3$, the distance between $u_2$ and $u_i$ with $\left \lfloor\frac{n+1}{2}\right\rfloor+2\leq i\leq n$ is decreased by $1$, and the distances between all other vertex pairs are increased or remain unchanged.
Thus
\begin{eqnarray*}
&&\mu_{\alpha}(T')-\mu_{\alpha}(C_n)\\
&=&x^\top (D_{\alpha}(T')-D_{\alpha}(G))x\\
&\ge& \alpha\left(x_{u_2}^2+x_{u_1}^2\right)+2(1-\alpha)x_{u_2}x_{u_1}-\sum_{i=\left \lfloor\frac{n+1}{2}\right\rfloor+2}^{n}\left(\alpha\left(x_{u_2}^2+x_{u_i}^2\right)+2(1-\alpha)x_{u_2}x_{u_i}\right)\\
&&+\sum_{i=3}^{\left \lceil\frac{n+1}{2}\right\rceil}(n-2i+3)\left(\alpha\left(x_{u_2}^2+x_{u_i}^2\right)+2(1-\alpha)x_{u_2}x_{u_i}\right)\\
&=&2x_{u_1}^2\left(1-\left(n-\left \lfloor\frac{n+1}{2}\right\rfloor-1\right)+\sum_{i=3}^{\left \lceil\frac{n+1}{2}\right\rceil}(n-2i+3)\right)\\
&=&2x_{u_1}^2\left (1+\left (n-1-\left\lceil\frac{n+1}{2}\right\rceil\right )\left(\left \lceil\frac{n+1}{2}\right\rceil -2\right)\right )\\
&\ge&2x_{u_1}^2\\
&>&0,
\end{eqnarray*}
and therefore  $\mu_{\alpha}(C_n)<\mu_{\alpha}(B_{n,3})$, as desired.
\end{proof}

A clique of $G$ is a subset of vertices whose induced subgraph is a complete graph, and the clique number of $G$ is the maximum number of vertices in a clique of $G$.
For $2\le \omega\le n$. Let  $Ki_{n,\omega}$ be the graph obtained from a complete graph $K_{\omega}$ and a path $P_{n-\omega}$ by adding an edge between a vertex of $K_\omega$ and a terminal vertex of $P_{n-\omega}$ if $\omega<n$ and let $Ki_{n,\omega}=K_n$ if $\omega=n$. In particular, $Ki_{n,2}\cong P_n$ for $n\ge 2$. The following result for $\alpha=0,\frac{1}{2}$ was given in \cite{NP,LL}.

\begin{Theorem}\label{888}
Let $G$ be a connected graph of order $n\geq2$ with clique number $\omega\geq2$. Then $\mu_{\alpha} (G)\leq\mu_{\alpha} (Ki_{n,\omega})$  with equality if and only if $G\cong Ki_{n,\omega}$.
\end{Theorem}

\begin{proof}  It is trivial if $\omega=n$ and it follows from Theorem~\ref{gen} if $\omega=2$.
Suppose that $3\le \omega\le n-1$. Let  $G$ be a  graph with maximum distance $\alpha$-spectral radius among connected graphs of order $n$ with clique number $\omega$.  We only need to show that $G\cong Ki_{n,\omega}$.

Let $S=\{v_1,\dots ,v_\omega\}$ be a clique of $G$. By Lemma~\ref{ad}, $G-E(G[S])$ is a forest. Let $T_i$ be the component of $G-E(G[S])$ containing $v_i$, where $1\leq i\leq \omega$. For $1\leq i\leq \omega$, by Corollary~\ref{pq},  if $T_i$ is nontrivial, then $T_i$ is a pendant path at $v_i$.
 Note that any two distinct vertices in $G[S]$ are adjacent.
By Corollary~\ref{adpq},  there is  only one nontrivial $T_i$, and thus $G\cong Ki_{n,\omega}$.
\end{proof}

Recall that $Ki_{n,3}$ is the unique unicyclic graph of order $n\ge 3$ with maximum distance $0$-spectral radius \cite{Yu2}, and the unique odd-cycle unicyclic graph of order $n\ge 3$ with maximum distance $\frac{1}{2}$-spectral radius \cite{LZ}.

\begin{Theorem}\label{odd cycle}
Let $G$ be a unicyclic odd-cycle graph of order $n\ge 3$. Then $\mu_{\alpha} (G)\leq \mu (Ki_{n,3})$ with equality if and only if $G\cong Ki_{n,3}$.
\end{Theorem}

\begin{proof}  If $n=3,4$, the result is trivial.
Suppose that $n\geq5$. Let $G$ be a graph with maximum distance $\alpha$-spectral radius among unicyclic odd-cycle graphs of order $n$. We only need to show that
$G\cong Ki_{n,3}$.

Let $C=v_1 \dots v_{2k+1}v_1$ be the unique cycle of $G$, where $k\ge1$. Let $T_i$ be the component of $G-E(C)$ containing $v_i$ for $1\le i\le 2k+1$. Let $U_1= V(T_{2k})\cup V(T_{2k+1})$, $U_2=\cup_{k+1\le i\le 2k-1} V(T_{i})$ and $U_3=\cup_{1\le i\le k-1} V(T_{i})$. Let $x$ be the distance $\alpha$-Perron vector of $G$.
Let \[\Gamma=\sum_{u\in U_1 }\sum_{v\in U_3}\left(\alpha\left(x_u^2+x_v^2\right)+2(1-\alpha)x_u x_v\right)-\sum_{u\in U_1 }\sum_{v\in U_2}\left(\alpha\left(x_u^2+x_v^2\right)+2(1-\alpha)x_u x_v\right).\]

Suppose that $k\ge 2$. Let $G'=G-v_1v_{2k+1}+v_{2k+1}v_{2k-1}$. Note that the length of $C$ is odd.
As we pass from $G$ to $G'$, the distance between a vertex in $S_1$ and a vertex in $S_3$ is increased by at least $1$, the distance between $S_2$ and $V(T_{2k+1})$ is decreased by $1$, and the distance between all other vertex pairs are increased or remain unchanged.
Thus
\begin{eqnarray*}
\mu_{\alpha}(G')-\mu_{\alpha}(G)&\ge & x^\top (D_{\alpha}(G')-D_{\alpha}(G))x \\
&\ge &\sum_{u\in U_1}\sum_{v\in U_3}\left(\alpha\left(x_u^2+x_v^2\right)+2(1-\alpha)x_u x_v\right)\\
&&-\sum_{u\in V(T_{2k+1})}\sum_{v\in U_2}\left(\alpha\left(x_u^2+x_v^2\right)+2(1-\alpha)x_u x_v\right)\\
&>&\sum_{u\in U_1}\sum_{v\in U_3}\left(\alpha\left(x_u^2+x_v^2\right)+2(1-\alpha)x_u x_v\right)\\
&&-\sum_{u\in U_1}\sum_{v\in U_2}\left(\alpha\left(x_u^2+x_v^2\right)+2(1-\alpha)x_u x_v\right).
\end{eqnarray*}
If $\Gamma\ge 0$, then $\mu_{\alpha}(G')>\mu_{\alpha}(G)$, a contradiction. Thus $\Gamma<0$.
Let $G''=G-v_{2k}v_{2k-1}+v_{2k}v_1$. As we pass from $G$ to $G''$, the distance between a vertex in $S_1$ and a vertex in $U_2$ is increased by at least $1$, the distance between $U_3$ and $V(T_{2k})$ is decreased by $1$, and the distance between all other vertex pairs are increased or remain unchanged.
As  above, we have
\begin{eqnarray*}
\mu_{\alpha}(G'')-\mu_{\alpha}(G)&\ge & x^\top(D_{\alpha}(G'')-D_{\alpha}(G))x \\
&\ge&\sum_{u\in U_1}\sum_{v\in U_2}\left(\alpha\left(x_u^2+x_v^2\right)+2(1-\alpha)x_u x_v\right)\\
&&-\sum_{u\in V(T_{2k})}\sum_{v\in U_3}\left(\alpha\left(x_u^2+x_v^2\right)+2(1-\alpha)x_u x_v\right)\\
&>&\sum_{u\in U_1}\sum_{v\in U_2}\left(\alpha\left(x_u^2+x_v^2\right)+2(1-\alpha)x_u x_v\right)\\
&&-\sum_{u\in U_1}\sum_{v\in U_3}\left(\alpha\left(x_u^2+x_v^2\right)+2(1-\alpha)x_u x_v\right)\\
&>& 0.
\end{eqnarray*}
Thus  $\mu_{\alpha}(G'')>\mu_{\alpha}(G)$, also a contradiction. It follows that $k=1$, i.e., the unique cycle of $G$ is of length $3$.

Obviously, $T_i$ is a tree for $1\le i\le 3$. For $1\le i\le 3$, by Corollary~\ref{pq}, if $T_i$ is nontrivial, then it is a path with a terminal vertex $v_i$. Then by Corollary~\ref{adpq}, only one $T_i$ is nontrivial. Thus $G\cong Ki_{n,3}$.
\end{proof}

%

\section{Remarks}

Some spectral properties of $D_{\alpha}(G)$ have been established in \cite{CHT}.
In this paper, we study the distance $\alpha$-spectral radius of a connected graph. 
We consider bounds for the distance $\alpha$-spectral radius, local transformations to change the distance $\alpha$-spectral radius, and the characterizations for graphs with minimum and/or maximum distance $\alpha$-spectral radius in some classes of connected graphs.
Lots of results in the literature are generalized and/or improved.

Besides the distance $\alpha$-spectral radius, we may concern other eigenvalues of $D_{\alpha}(G)$ for a connected graph $G$. We give  examples.

For an $n\times n$ Hermitian matrix $A$, let $\lambda_1(A), \dots, \lambda_n(A)$ be the eigenvalues, arranged in a non-increasing order.

\begin{Lemma}\label{CW}\cite{CR} Let $A$, $B$ be $n\times n$ Hermitian matrices. Then
\[
\lambda_{j}(A+B)\le \lambda_{i}(A)+\lambda_{j-i+1}(B) \mbox{ for $1\le i\le j\le n$},
\]
and
\[
\lambda_{j}(A+B)\ge \lambda_{i}(A)+\lambda_{j-i+n}(B) \mbox{ for $1\le j\le i\le n$}.
\]
\end{Lemma}

As in the recent work of Atik and Panigrahi~\cite{AP}, we have

\begin{Theorem} Let $G$ be a connected graph and $\lambda$ be any eigenvalue of $D_{\alpha}(G)$ other than the distance $\alpha$-spectral radius. Then \[
2\alpha T_{\min}(G)-T_{\max}(G)+(1-\alpha)(n-2)\le \lambda \le T_{\max}(G)-(1-\alpha)n.
\]
\end{Theorem}

\begin{proof}  Let $D_{\alpha}(G)=A+B$, where $A=(\alpha T_{\min}(G)-(1-\alpha))I_n+(1-\alpha)J_{n\times n}$. Then $B$ is a nonnegative symmetric matrix with maximum row sum $T_{\max}(G)-\alpha T_{\min}(G)-(1-\alpha)(n-1)$. Thus $|\lambda_{n}(B)|\le \lambda_{1}(B)\le T_{\max}(G)-\alpha T_{\min}(G)-(1-\alpha)(n-1)$.

 For matrix $A$, we have $\lambda_{1}(A)=\alpha T_{\min}(G)+(1-\alpha)(n-1)$ and $\lambda_{j}(A)=\alpha T_{\min}(G)-1+\alpha$ for $j=2,\dots,n$. For $j=2,\dots, n$, we have by Lemma~\ref{CW} that
\begin{eqnarray*}
\lambda_{j}(D_{\alpha}(G)) &\le & \lambda_{1}(B)+\lambda_{j}(A)\\
&\le & T_{\max}(G)-\alpha T_{\min}(G)-(1-\alpha)(n-1)+\alpha T_{\min}(G)-1+\alpha\\
&=&T_{\max}(G)-(1-\alpha)n.
\end{eqnarray*}
Similarly, for $j=2,\dots,n$,
\begin{eqnarray*}
\lambda_{j}(D_{\alpha}(G)) &\ge & \lambda_{n}(B)+\lambda_{j}(A)\\
&\ge & -T_{\max}(G)+\alpha T_{\min}(G)+(1-\alpha)(n-1)+\alpha T_{\min}(G)-1+\alpha\\
&=& 2\alpha T_{\min}(G)-T_{\max}(G)+(1-\alpha)(n-2).
\end{eqnarray*}
This completes the proof.
\end{proof}

Let $G$ be a connected graph and $\lambda$ be any eigenvalue of $D_{\alpha}(G)$ other than the distance $\alpha$-spectral radius. By previous theorem, we have \[
|\lambda| \le T_{\max}(G)-(1-\alpha)(n-2).
\]

The distance $\alpha$-energy of a connected graph $G$ of order $n$ is defined as
\[
\mathcal{E}_{\alpha}(G)=\sum_{i=1}^n\left|\mu_{\alpha}^{(i)}(G)-\frac{2\alpha\sigma(G)}{n}\right|.
\]
Obviously, $\mathcal{E}_0(G)$ is the distance energy of $G$~\cite{In,ZI}, while
\[
\mathcal{E}_{1/2}(G)=\frac{1}{2}\sum_{i=1}^n\left|2\mu_{1/2}^{(i)}(G)-\frac{2\sigma(G)}{n}\right|
\]
is half of the distance signless Laplacian energy of $G$~\cite{DAH}.
Thus, it is possible to study the distance energy and the distance signless Laplacian energy in a unified way.

\vspace{3mm}
\noindent {\bf Acknowledgement.} This work was supported by the National Natural Science Foundation of China (No.~11671156).

\end{document}